\documentclass[a4paper]{amsart}[11pt]
\usepackage{latexsym,amssymb,amsthm}
\usepackage{color}
\usepackage{dsfont}
\usepackage{fancyhdr}
\usepackage{nicematrix}
\usepackage{tikz}
\usepackage{hyperref}
\usepackage{float}
\usepackage{enumitem}
\usepackage{tabularx}
\def\doctype{}

\renewcommand\L{\mathrm{L}}

\newcommand\M{\mathrm{M}}

\renewcommand\S{\mathrm{S}}

\newcommand{\comment}[1]{}

\newcommand{\G}{\mathcal{G}}
\newcommand{\magma}{{\sc Magma}}
\newcommand{\gap}{{\sf GAP}}

\fancypagestyle{titlepage}{
\fancyhead[R]{\doctype}
\fancyhead[CL]{}
\cfoot{\vspace{5pt} \thepage}
}

\let\oldsection\section
\newcommand\boldsection[1]{\oldsection{\bf #1}}
\newcommand\starsection[1]{\oldsection*{\bf #1}}
\makeatletter
\renewcommand\section{\@ifstar\starsection\boldsection}
\makeatother

\newtheoremstyle{algorithm}
  {12pt}		  
  {0pt}  
  {\tt}  
  {\parindent}     
  {\bf}  
  {. }    
  {\newline}    
  {}     
\theoremstyle{algorithm}

\newtheoremstyle{theorem}
  {12pt}		  
  {0pt}  
  {\sl}  
  {\parindent}     
  {\bf}  
  {. }    
  { }    
  {}     
\theoremstyle{theorem}
\newtheorem{thm}{Theorem}[section]  
\newtheorem{lemma}[thm]{Lemma}     
\newtheorem{cor}[thm]{Corollary}
\newtheorem{conj}[thm]{Conjecture}

\newtheorem{prop}[thm]{Proposition}

\newtheoremstyle{definition}
  {12pt}		  
  {0pt}  
  {}  
  {\parindent}     
  {\bf}  
  {. }    
  { }    
  {}     
\theoremstyle{definition}

\renewcommand{\proofname}{Proof}

\makeatletter
\renewenvironment{proof}[1][\proofname]{\par
  \pushQED{\qed}%
  \normalfont \partopsep=\z@skip \topsep=\z@skip
  \trivlist
  \item[\hskip\labelsep
        \scshape
    #1\@addpunct{.}]\ignorespaces
}{%
  \popQED\endtrivlist\@endpefalse
}
\makeatother

\makeatletter
\renewcommand*\@maketitle{%
  \normalfont\normalsize
  \@adminfootnotes
  \@mkboth{\@nx\shortauthors}{\@nx\shorttitle}%
  \global\topskip0\p@\relax 
  \@settitle
  \ifx\@empty\authors \else {\vskip 1em
\vtop{\centering\shortauthors\@@par}} \fi
  \ifx\@empty\@date \else {\vskip 1em \vtop{\centering\@date\@@par}}\fi 
  \ifx\@empty\@dedicatory
  \else
    \baselineskip18\p@
    \vtop{\centering{\footnotesize\itshape\@dedicatory\@@par}%
      \global\dimen@i\prevdepth}\prevdepth\dimen@i
  \fi
  \@setabstract
  \normalsize
  \if@titlepage
    \newpage
  \else
    \dimen@34\p@ \advance\dimen@-\baselineskip
    \vskip\dimen@\relax
  \fi
} 
\renewcommand*\@adminfootnotes{%
  \let\@makefnmark\relax  \let\@thefnmark\relax
  \ifx\@empty\@subjclass\else \@footnotetext{\@setsubjclass}\fi
  \ifx\@empty\@keywords\else \@footnotetext{\@setkeywords}\fi
  \ifx\@empty\thankses\else \@footnotetext{%
    \def\par{\let\par\@par}\@setthanks}%
  \fi
\thispagestyle{titlepage}
}
\makeatother

\pgfsetlayers{nodelayer,edgelayer} 
\usepackage{graphicx}
\begin{document}
\title[]{\large Greedy base sizes \\for sporadic simple groups}

\author{Coen del Valle}
\address{
School of Mathematics and Statistics,
University of St Andrews, St Andrews, UK
}
\email{cdv1@st-andrews.ac.uk}

\thanks{The author is grateful to his supervisors Colva Roney-Dougal and Peter Cameron for their support and guidance. The author would also like to thank Tim Burness for his helpful email communication. Research of Coen del Valle is supported by the Natural Sciences and Engineering Research Council of Canada (NSERC), [funding reference number PGSD-577816-2023], as well as a University of St Andrews School of Mathematics and Statistics Scholarship.}

\date{\today}

\begin{abstract}
A base for a permutation group $G$ acting on a set $\Omega$ is a sequence $\mathcal{B}$ of points of $\Omega$ such that the pointwise stabiliser $G_{\mathcal{B}}$ is trivial. Denote the minimum size of a base for $G$ by $b(G)$. There is a natural greedy algorithm for constructing a base of relatively small size; denote by $\mathcal{G}(G)$ the maximum size of a base it produces. Motivated by a long-standing conjecture of Cameron, we determine $\G(G)$ for every almost simple primitive group $G$ with socle a sporadic simple group, showing that $\mathcal{G}(G)=b(G)$.
\end{abstract}

\maketitle
\hrule

\bigskip

\section{Introduction}
Let $G\leq \mathrm{Sym}(\Omega)$ for some finite set $\Omega$. A \emph{base} for $G$ is a sequence $\mathcal{B}$ of points of $\Omega$, such that the pointwise stabiliser $G_{\mathcal{B}}$ is trivial. Bases are integral objects in modern computational group theory (see e.g.~\cite{ser}), with the complexity of many algorithms depending heavily on the size of base used. Accordingly, finding small bases has been a topic of significant interest for more than 30 years; the size of a smallest possible base for $G$ is called the \emph{base size} of $G$, and is denoted $b(G)$.

In general, computing the base size of an arbitrary permutation group is a difficult task. Indeed, in 1992, Blaha~\cite{blaha} showed that the problem of determining the base size of a group is NP-hard. Despite this, there are many well-known asymptotically sharp bounds and precise formulae for certain families of permutation groups. For example, the base size of every almost simple primitive group with socle not of Lie type has been determined recently as a result of work spanning almost two decades~\cite{burn2,bspo,cdvrd,MeSp,ms}. Moreover, those with socle of Lie type which are not classical groups in subspace actions are known to have base size at most 6, with the precise base size known in many cases~\cite{bls}.

Several of the aforementioned results have non-constructive proofs, and so finding a minimum base remains a challenging problem. To this end, there is a natural greedy algorithm  first introduced by Blaha~\cite{blaha} for constructing an explicit base of small size: set $\mathcal{B}_1=(\beta_1)$ where $\beta_1$ is a point in a largest $G$-orbit, and for $i\geq 2$ let $\mathcal{B}_i:=(\beta_1,\beta_2,\dots,\beta_i)$, where $\beta_i$ is a point in a largest orbit of $G_{\mathcal{B}_{i-1}}$. Since $\Omega$ is finite, there is some $j\geq1$ minimal such that $\mathcal{B}_j$ is a base --- called a \emph{greedy base} --- for $G$. By the orbit-stabiliser lemma, choosing points in this way guarantees that $G_{\mathcal{B}_i}$ is as small as possible relative to $G_{\mathcal{B}_{i-1}}$, thus one should expect to construct a small base. Let $\mathcal{G}(G)$ denote the maximum size of a greedy base for $G$. Blaha~\cite{blaha} showed that the greedy approximation is nearly sharp, that is, there is some absolute constant $d$ such that $$\mathcal{G}(G)\leq db(G)\log \log |\Omega|,$$
and moreover, for any $k\geq 2$ and $n$ sufficiently large, there exists a group $G$ of degree $n$ such that $b(G)=k$ and $\mathcal{G}(G)\geq\frac{1}{5}k\log\log n$. In the case of primitive groups, more is conjectured to be true.

\begin{conj}[Cameron's greedy conjecture~\cite{cam}]\label{cc}
There is some absolute constant $c$ such that if $G$ is a finite primitive permutation group then $\mathcal{G}(G)\leq cb(G)$.
\end{conj}
Cameron's greedy conjecture has remained open for the past 25 years, although recent work has begun making progress towards answering it in the affirmative~\cite{cdvrd2}. In this paper we determine precisely the greedy base size of every almost simple primitive group with sporadic socle and show that the base size and the greedy base size coincide. Although this analysis may not be strictly essential to the proof of Cameron's greedy conjecture, it is a key step towards determining the optimal constant $c$. 
\begin{thm}\label{main}
Let $G$ be an almost simple primitive group with sporadic socle. Then $\G(G)=b(G)$.
\end{thm}
\begin{rk} The base sizes for such groups are known and can be found in~\cite{bspo,Bfin}.\end{rk}

Our methods are a mixture of computational and probabilistic techniques. We make significant use of the computer algebra systems \gap~\cite{gap} and \magma~\cite{magma}, and introduce a probabilistic technique which is used as an extension to the marvelous method of fixed point ratios introduced by Liebeck and Shalev~\cite{LieSha} in their celebrated proof of the Cameron--Kantor conjecture~\cite{ck}.

The structure of the paper is straightforward. In Section~\ref{tools} we familiarise the reader with the standard computational techniques to be used in proving Theorem~\ref{main}, and we also prove two probabilistic results which will be essential to our study. Finally, Section~\ref{pmain} is dedicated to using the tools introduced in Section~\ref{tools} to prove Theorem~\ref{main}. Throughout this paper we use the standard $\mathbb{ATLAS}$ notation~\cite{atlas}.
\begin{section}{Techniques and preliminaries}\label{tools}
In this section we cover some of the preliminary theory which is needed for the proof of Theorem~\ref{main}. We begin by summarising the base sizes of almost simple primitive groups with sporadic socle.

\begin{thm}[{\cite[Theorem 1]{bspo}, \cite[Theorem 1]{Bfin}}] Let $G$ be an almost simple primitive group with sporadic socle. Suppose $G$ has point stabiliser $H$, and that $b(G)\ne 2$. Then either $(G,H,b(G))$ is displayed in \cite[Tables 1 and 2]{bspo} or $(G,H,b(G))=(\mathbb{B},2^{2+10+20}.(\M_{22}:2\times \S_3), 3)$.
\end{thm}
Our results are obtained through directly calculating the greedy base sizes, and comparing these with the values in the tables of Burness, O'Brien, and Wilson~\cite{bspo}. In the case of groups with base size 2, the following observation is immediate.
\begin{lemma}\label{b=2}
Let $G$ be a transitive group with $b(G)=2$. Then $\mathcal{G}(G)=2$.
\end{lemma}
\begin{proof}
The group $G$ has a regular suborbit.
\end{proof} 
For the remainder of the paper we consider groups with base size at least 3. We have three main techniques for such groups which we describe now. 

Let $G$ be a faithful transitive permutation group with point stabiliser $H$. If both the degree and order of $G$ are sufficiently small, and a permutation representation of $G$ is available (or can be constructed) in \magma~\cite{magma}, we can simply run the greedy algorithm with a backtrack search to determine the quantity $\G(G)$. The other two techniques are probabilistic and are more or less the same as each other. If $b(G)=3$ and the degree of $G$ is sufficiently small, we can often prove that $G$ has a suborbit which consists of a large proportion of the total points; we will see in Proposition~\ref{key} that this is enough to deduce that $\G(G)=3$. On the other hand, if $b(G)=3$ but the degree of $G$ is large, then $|H|$ will be small enough that there is some conjugate $H^g$ such that $|H\cap H^g|$ can be computed in \magma~\cite{magma} --- in this case we run a random search until finding a conjugate such that $|H\cap H^g|$ is sufficiently small, which allows us to deduce that $\G(G)=3$ from Corollary~\ref{int}.

Before introducing our probabilistic techniques we describe the powerful method first introduced by Liebeck and Shalev~\cite{LieSha} upon which our approach is based. Let $G$ be a permutation group acting faithfully on a set $\Omega$. The \emph{fixed point ratio} of $x\in G$ is the ratio $$\mathrm{fpr}(x,\Omega):=\frac{|\{\alpha\in\Omega : \alpha^x=\alpha\}|}{|\Omega|}.$$ That is, $\mathrm{fpr}(x,\Omega)$ is the probability that a uniformly chosen point $\alpha\in\Omega$ is fixed by $x$. Note that if $G$ is transitive with point stabiliser $H$ then $$\mathrm{fpr}(x,\Omega)=\mathrm{fpr}(x,G/H)=\frac{|x^G\cap H|}{|x^G|}.$$ Now, let $Q(G,c)$ be the probability that a uniformly chosen $c$-tuple of points of $\Omega$ is not a base for the transitive group $G$. Let $\mathcal{P}$ be the set of prime order elements of $G$, and $x_1,x_2,\dots, x_k$ representatives of the distinct classes of prime order elements in $G$. If a $c$-tuple is not a base then it is fixed by an element of prime order and so $$Q(G,c)\leq \sum_{x\in \mathcal{P}}\mathrm{fpr}(x,\Omega)^c=\sum_{i=1}^k |x_i^G|\cdot\left(\frac{|x_i^G\cap H|}{|x_i^G|}\right)^c=:\hat{Q}(G,c).$$ An immediate consequence is the following lemma.

\begin{lemma}\label{Q}
Let $G$ be a finite transitive group and let $c$ be a positive integer. If $\hat{Q}(G,c)<1$ then $b(G)\leq c$.
\end{lemma}

It is the above key observation that we extend to greedy bases. Before formally describing our main tool, we begin by providing the intuition. Suppose that $G$ is a group acting transitively on some set $\Omega$ such that $b(G)=3$. If it is very likely that a uniformly chosen 3-tuple $(\alpha_1,\alpha_2,\alpha_3)$ has at least one of $\alpha_2$ or $\alpha_3$ in some fixed largest $G_{\alpha_1}$-orbit --- that is, if $G$ has a very large greatest subdegree --- and also likely that $(\alpha_1,\alpha_2,\alpha_3)$ form a base, then with non-zero probability both will occur simultaneously. If both events occur then some permutation of $(\alpha_1,\alpha_2,\alpha_3)$ forms a greedy base.

\begin{prop}\label{key}
Let $G$ be a group acting transitively and faithfully on a finite set $\Omega$ with point stabiliser $H$. Suppose that $\hat{Q}(G,3)<1$ and that $G$ has a subdegree which is greater than $$\mathcal{D}(G,H):=|G:H|\cdot\left(1-\sqrt{1-\hat{Q}(G,3)}\right)$$ 
then $\mathcal{G}(G)\leq3$.
\end{prop}
\begin{proof}
Suppose $G$ has largest subdegree $d>\mathcal{D}(G,H)$. For each $\alpha\in\Omega$ let $O_1(\alpha),O_2(\alpha),\dots,O_k(\alpha)$ be the $G_{\alpha}$-orbits of size $d$. As established, the probability that a uniformly selected 3-tuple is a base is at least $1-\hat{Q}(G,3)$. Fix some $1\leq i \leq k$, and let $p_i$ be the probability that a uniformly chosen 3-tuple $(\alpha_1,\alpha_2,\alpha_3)$ is such that at least one of $\alpha_2$ or $\alpha_3$ is in $O_i(\alpha_1)$. Then $$p_i=\frac{2d}{|G:H|}-\left(\frac{d}{|G:H|}\right)^2>2\left(1-\sqrt{1-\hat{Q}(G,3)}\right)-\left(1-\sqrt{1-\hat{Q}(G,3)}\right)^2=\hat{Q}(G,3).$$ Now, note that $$p_i+(1-\hat{Q}(G,3))> \hat{Q}(G,3)+(1-\hat{Q}(G,3))=1,$$ so there must exist some 3-tuple $(\alpha_1,\alpha_2,\alpha_3)$ which is a base but also satisfies $\alpha_2\in O_i(\alpha_1)$. In particular, $G_{\alpha_1,\alpha_2}$ has a regular orbit. Since $i$ was arbitrary it follows that if $\alpha,\beta\in\Omega$ are such that $\beta$ is in a largest $G_{\alpha}$-orbit, then $G_{\alpha,\beta}$ has a regular orbit. Therefore, $\mathcal{G}(G)=3$, as desired.
\end{proof}
Proposition~\ref{key} is most useful when we can easily obtain information on the subdegrees of $G$ --- when this is not possible we can often use a random search to find a small 2-point stabiliser. The following corollary proves useful in this case, and is immediate from Proposition~\ref{key} and the fact that point stabilisers in transitive groups are conjugate.
\begin{cor}\label{int}
Let $G$ be a finite group acting transitively and faithfully with point stabiliser $H$. Suppose that $\hat{Q}(G,3)<1$ and that there is some $g\in G$ such that $$|H\cap H^g|<\mathcal{S}(G,H):=\frac{|H|^2}{|G|\left(1-\sqrt{1-\hat{Q}(G,3)}\right)}.$$ Then $\mathcal{G}(G)\leq3$.
\end{cor}
\begin{rk}
Importantly, to compute the quantities $\mathcal{D}(G,H)$ and $\mathcal{S}(G,H)$ it is sufficient to know the sizes of conjugacy classes of elements of prime order in $G$ as well as those of $H$ and their fusion in $G$ --- in most cases this is accessible directly from the \gap~Character Table Library~\cite{ctlib}.
\end{rk}

Of course Proposition~\ref{key} and Corollary~\ref{int} only apply when $b(G)\leq 3$, however, we can run the greedy algorithm directly on all but one of the relevant groups with base size greater than 3.
\end{section}
\section{Proving Theorem~\ref{main}}\label{pmain}
In this section we compute the greedy base sizes of all primitive almost simple groups with sporadic socle and base size at least 3, hence proving Theorem~\ref{main}. As was hinted at the end of Section~\ref{tools}, most groups with base size greater than $3$ (and many with base size exactly 3) may easily be dealt with purely computationally.
\begin{prop}\label{smallindex}
Let $G$ be an almost simple primitive group with sporadic socle and point stabiliser $H$. Then either $b(G)=\mathcal{G}(G)$ or $(G,H,b(G))$ is listed in Table~\ref{tab1}.
\end{prop}
\begin{proof}
If $(G,H)$ is not listed in Table~\ref{tab1}, then either \begin{itemize}\item$b(G)=2$ and so $\G(G)=2$ by Lemma~\ref{b=2};\item  the permutation representation of $G$ on $G/H$ exists in \magma~\cite{magma}; or\item some permutation representation of $G$ exists in \magma~\cite{magma} and $H$ may be constructed either through the command {\texttt{MaximalSubgroups(G)}}, or through a straight-line program available from the Web Atlas~\cite{webatlas}.\end{itemize} In the final case, we obtain the permutation representation of $G$ on $G/H$ in \magma~by using the command {\texttt{CosetAction(G,H)}}. We run the greedy algorithm with a backtrack search on the appropriate permutation representation of $G$ to deduce the result.
\end{proof}

The groups of Table~\ref{tab1} have degree too large to be easily dealt with using the explicit greedy algorithm; from here onward, we use the probabilistic techniques developed in Section~\ref{tools}. Before stating the next result we remind the reader of some elementary character theory.

Oftentimes in order to deduce that a permutation group has a large enough suborbit it is sufficient to know that it has few suborbits; we can calculate the number of suborbits --- called the \emph{rank} --- using character theory. Indeed, given the permutation character $\chi$ of a group $G$ acting on its coset space $G/H$, the rank of $G$ is precisely $\langle \chi,\chi\rangle$, where $\langle -,-\rangle$ is the standard scalar product on class functions. Moreover, such permutation characters can be computed as the induction $1^G_H$ to $G$ of the trivial character of $H$. It follows that with the help of the \gap ~Character Table Library~\cite{ctlib} we can find the rank of $G$ with very little effort.

To illustrate some of our techniques in practice we give a descriptive proof in the case that $G=\mathrm{Fi}_{23}$.

\begin{prop}
Let $G=\mathrm{Fi}_{23}$ acting primitively with point stabiliser $H$. Then $b(G)=\mathcal{G}(G)$.
\end{prop}
\begin{proof}
By Proposition~\ref{smallindex} we may assume that $H\in\{\S_8(2),\mathrm{O}_7(3)\times S_3,2^{11}.\M_{23},3^{1+8}.2^{1+6}.3^{1+2}.2S_4\}.$ Note that the character tables of $G$ and $H$ are available in the \gap~Character Table Library~\cite{ctlib}. Moreover, the class fusion is known and the fusion maps are available in~\gap,~whence $\mathcal{D}(G,H)$ may be computed precisely. We now handle each possible stabiliser $H$ separately.

If $H\in\{\S_8(2),2^{11}.\M_{23}\}$ then $\mathcal{D}(G,H)<10^7$, but in both cases the subdegrees have been determined by Linton, Lux, and Soicher ~\cite{linton}. In particular, if $H=\S_8(2)$ then $G$ has a suborbit of length $32901120>10^7$, and if $H=2^{11}.\M_{23}$ then $G$ has a suborbit of length $58032128>10^7$. Therefore, by Proposition~\ref{key}, in both cases we deduce that $\G(G)=3$.

Now, suppose $H=\mathrm{O}_7(3)\times S_3$. Then we calculate $\mathcal{D}(G,H)<10^7$, and that $\langle 1_H^G,1_H^G\rangle=14$. Therefore, $G$ has a subdegree which is at least $(|G:H|-1)/(14-1)=11434043>10^7$, thus $\G(G)=3$ by Proposition~\ref{key}.

Finally, if $H=3^{1+8}.2^{1+6}.3^{1+2}.2S_4$ then $\mathcal{D}(G,H)<23000000$ and $\langle 1_H^G,1_H^G\rangle=36$. Thus $G$ has a subdegree which is at least $(|G:H|-1)/35>35000000>23000000$, so the result follows from Proposition~\ref{key}.
\end{proof}
\begin{table}[H]
\centering

\begin{tabularx}{11.6cm}{llc|llc}
$G$ & $H$ & $b(G)$  &$G$ & $H$ & $b(G)$\\
\hline
$\mathrm{Fi}_{23}$&$\mathrm{S}_8(2)$&3		&$\mathrm{Fi}_{24}'$&$2.\mathrm{Fi}_{22}:2$&3\\
&$\mathrm{O}_7(3)\times S_3$&3	&&$(3\times \mathrm{O}_8^+(3):3):2$&3		\\
&$2^{11}.\M_{23}$&3&		&$\mathrm{O}_{10}^-(2)$&3\\
&$3^{1+8}.2^{1+6}.3^{1+2}.2S_4$&3&		&$3^7.\mathrm{O}_7(3)$&3		\\
&&&&$3^{1+10}:\mathrm{U}_5(2):2$&3
\\
$\mathrm{J}_4$&$2^{11} : \M_{24}$&3&&$2^{11}.\M_{24}$&3\\
&$2^{1+12}.3.\M_{22}:2$&3\\
&$2^{10}:\L_5(2)$&3	&$\mathrm{Fi}_{24}$&$(2\times2.\mathrm{Fi}_{22}):2$&3	\\
&&&&$S_3\times \mathrm{O}_8^+(3):S_3$&3\\
$\mathrm{Ly}$&$G_2(5)$&3&&$\mathrm{O}_{10}^-(2):2$&3\\
&$3.\mathrm{McL}:2$& 3&		&$3^7.\mathrm{O}_7(3):2$&3\\
&&&&$3^{1+10}:(\mathrm{U}_5(2):2\times2)$&3\\
$\mathrm{Co}_1$&$2^{1+8}.\mathrm{O}_8^+(2)$&3&&$2^{12}.\M_{24}$&3\\
&$\mathrm{U}_6(2):S_3$&3&&$(2\times2^2.\mathrm{U}_6(2)):S_3$&3\\
&$(A_4\times G_2(4)) : 2$&3\\
&$2^{2+12}:(A_8\times S_3)$&3&$\mathbb{B}$&$2.{}^2E_6(2):2$&4\\
&$2^{4+12}.(S_3\times3.S_6)$&3&&$2^{1+22}.\mathrm{Co}_2$&3\\
&&&&$\mathrm{Fi}_{23}$&3\\
$\mathrm{HN}$& $2.\mathrm{HS}.2$&3&&$2^{9+16}.\mathrm{S}_8(2)$&3\\
&$\mathrm{U}_3(8):3$&3&&$\mathrm{Th}$&3\\
&&&&$(2^2\times F_4(2)):2$&3\\
$\mathrm{HN.2}$& $4.\mathrm{HS}.2$&3&&$2^{2+10+20}.(\M_{22}:2\times \S_3)$&3\\
&$\mathrm{U}_3(8):6$&3\\
&&&$\mathbb{M}$&$2.\mathbb{B}$&3\\
$\mathrm{Th}$&${}^3D_4(2):3$&3\\
&$2^5.\mathrm{L}_5(2)$&3\\
\end{tabularx}
\caption{The groups not dealt with in Proposition~\ref{smallindex}}
\label{tab1}
\end{table}
We now present Table~\ref{tab2} which consists of groups for which Proposition~\ref{key} may easily be applied using our methods so far. Table~\ref{tab2} has entries $(G,H,d,a,\text{ref.})$; here $d$ is an upper bound on $\mathcal{D}(G,H)$ and $a$ is a lower bound on the size of a largest suborbit of $G$. The quantity $a$ is obtained from the technique or reference in the fifth column, where `rank' refers to the calculation  $$(|G:H|-1)/(\langle 1_H^G,1_H^G\rangle-1).$$ Since each entry $(G,H)$ of Table~\ref{tab2} is such that $a>d$ we immediately deduce the following result.

\begin{prop}\label{tabprop1}
Suppose $(G,H)$ are listed in Table~\ref{tab2}. Then $\mathcal{G}(G)=3$.
\end{prop}
As the techniques needed are similar to what we have done so far, we now complete the determination of primitive actions of $G=\mathbb{B}$ with $\G(G)=3$.
\begin{prop}
Let $G=\mathbb{B}$ acting primitively with point stabiliser $H$ and suppose that $b(G)=3$. Then $\G(G)=3$.
\end{prop}
\begin{proof}
By Propositions~\ref{smallindex} and \ref{tabprop1} we may assume that ${H\in\{(2^2\times F_4(2)):2,2^{2+10+20}.(\M_{22}:2\times \S_3)\}}$. If $H=(2^2\times F_4(2)):2$, the fusion map is not known, but we can compute all possible class fusions using the \gap~\cite{gap} command {\tt{PossibleClassFusions()}}. We find that the fixed point ratios do not vary between each possible fusion map so we can compute $\mathcal{D}(G,H)$ as normal. This yields ${\mathcal{D}(G,H)\leq5.3\times10^{10}}$. Similarly, we deduce that $G$ has rank 163 by calculating the permutation character using the {\tt{InducedClassFunctionsByFusionMap()}} command. Therefore, $G$ has a suborbit of length at least $$(|G:H|-1)/162>9.6\times 10^{14}>\mathcal{D}(G,H),$$ whence $\G(G)=3$ by Proposition~\ref{key}.

On the other hand, suppose $H=2^{2+10+20}.(\M_{22}:2\times \S_3)$. Then $G$ has a suborbit of length $$11429423370731520=|H|/2$$ by \cite[Table 1]{Bfin}. In particular any smallest 2-point stabiliser in $G$ has order 2 and so has a regular orbit. Therefore, $\G(G)=3$ as desired.
\end{proof}
\begin{table}[H]
\centering

\begin{tabularx}{11.8cm}{lllll}
$G$ & $H$ & $d$  &$a$&ref.\\
\hline

$\mathrm{J}_4$&$2^{11} : \M_{24}$&$1.1\times10^7$&$8.2\times10^7$&\cite[Theorem 8.26]{iv}\\
&$2^{1+12}.3.\M_{22}:2$&$1.5\times10^6$&$1.9\times10^8$&rank\\
&$2^{10}:\L_5(2)$&$7.3\times10^5$	&$3\times 10^8$&rank	\\\\
$\mathrm{Ly}$&$G_2(5)$&$2.6\times 10^5$&$5.8\times10^6$&$\mathbb{ATLAS}$~\cite{atlas}\\
&$3.\mathrm{McL}:2$& $5.6\times10^5$&	$7.1\times10^6$	&$\mathbb{ATLAS}$~\cite{atlas}\\\\

$\mathrm{Co}_1$&$2^{1+8}.\mathrm{O}_8^+(2)$&$10^7$&$2.5\times10^7$&\cite[Table 2]{bates}\\
&$(A_4\times G_2(4)) : 2$&$3.4\times 10^6$&$2.9\times 10^7$&rank\\
&$2^{2+12}:(A_8\times S_3)$&$1.6\times 10^6$&$4.9\times 10^7$&rank\\
&$2^{4+12}.(S_3\times3.S_6)$&$1.7\times 10^6$&$5.3\times 10^7$&rank\\\\
$\mathrm{HN}$& $\mathrm{U}_3(8):3$&$7.1\times10^3$&$9.1\times10^5$&rank\\\\
$\mathrm{HN.2}$&$\mathrm{U}_3(8):6$&$9\times10^3$&$1.1\times10^6$&rank\\\\
$\mathrm{Th}$&${}^3D_4(2):3$&$3.5\times10^4$&$1.4\times10^7$&rank\\
&$2^5.\mathrm{L}_5(2)$&$2.1\times10^4$&$2.8\times10^7$&rank\\\\
$\mathrm{Fi}'_{24}$&$(3\times\mathrm{O}_{8}^+(3):3):2$&$5.1\times10^8$&$10^9$&rank\\
&$\mathrm{O}_{10}^-(2)$&$5.2\times10^8$&$3.1\times10^9$&rank\\
&$3^7.\mathrm{O}_7(3)$&$5.7\times 10^7$&$7.3\times10^9$&rank\\
&$3^{1+10}:\mathrm{U}_5(2):2$&$1.4\times10^8$&$9.5\times10^9$&rank\\
&$2^{11}.\M_{24}$&$3.3\times10^8$&$2.1\times10^{10}$&rank\\\\
$\mathrm{Fi}_{24}$&$\mathrm{O}_{10}^-(2):2$&$6\times10^8$&$3.1\times10^9$&rank\\
&$3^7.\mathrm{O}_7(3):2$&$1.1\times10^9$&$5.2\times10^{10}$&\cite[Theorem 6.8]{linton}\\
&$3^{1+10}:(\mathrm{U}_5(2):2\times2)$&$4.7\times10^9$&$9.5\times10^9$&rank\\
&$2^{12}.\M_{24}$&$3.3\times10^8$&$3.1\times10^{10}$&rank\\
&$(2\times2^2.\mathrm{U}_6(2)):S_3$&$1.1\times10^{10}$&$2.4\times10^{10}$&rank\\\\
$\mathbb{B}$&$2^{1+22}.\mathrm{Co}_2$&$3.1\times10^{12}$&$6.8\times 10^{12}$&\cite[Table 1]{mul1}\\
&$\mathrm{Fi}_{23}$&$6.1\times 10^{11}$&$2.8\times 10^{14}$&\cite[Table 2]{mul2}\\
&$2^{9+16}.\mathrm{S}_8(2)$&$8.4\times 10^{11}$&$8.4\times 10^{13}$&rank\\
&$\mathrm{Th}$&$2.5\times 10^8$&$1.3\times 10^{15}$&rank\\\\
$\mathbb{M}$&$2.\mathbb{B}$&$1.1\times10^{19}$&$1.2\times10^{19}$&rank\\
\end{tabularx}
\caption{Some groups for which Proposition~\ref{key} may easily be applied.}\label{tab2}
\end{table}
We now move on to use our random search technique. We first illustrate this method explicitly in the case $G=\mathrm{Co}_1$.
\begin{prop}
Let $G=\mathrm{Co}_1$ acting primitively with point stabiliser $H$. Then $\G(G)=b(G)$.
\end{prop}
\begin{proof}
By Propositions~\ref{smallindex} and \ref{tabprop1} we may assume that $H=\mathrm{U}_6(2):S_3$. The relevant character tables and fusion maps are available in the \gap~Character Table Library~\cite{ctlib} and so we may directly compute $\mathcal{S}(G,H)>4300$. Now, in \magma~\cite{magma} we compute {\tt{Order( H meet H\string^Random(G) )}} until eventually finding $g\in G$ such that $|H\cap H^g|=1536<4300$. The result follows from Corollary~\ref{int}.
\end{proof}
We now present Table~\ref{table3} which includes the remaining groups for which the technique of the above result proves fruitful. The entries of the table are $(G,H,s,|H\cap H^g|)$. The entry $s$ is a lower bound on the quantity $\mathcal{S}(G,H)$, and $|H\cap H^g|$ is the order of some intersection of conjugates of $H$ found using a random search.

\begin{table}[h]
\centering
\begin{tabularx}{7.1cm}{llll}
$G$ & $H$ & $s$  &$|H\cap H^g|$\\
\hline
$\mathrm{HN}$&$2.\mathrm{HS}.2$&790&240\\\\
$\mathrm{HN}.2$&$4.\mathrm{HS}.2$&1100&480\\\\
$\mathrm{Fi}'_{24}$&$2.\mathrm{Fi}_{22}:2$&390000&93312\\\\
$\mathrm{Fi}_{24}$&$(2\times2.\mathrm{Fi}_{22}):2$&200000&186624\\
&$S_3\times\mathrm{O}_8^+(3):S_3$&49000&27648\\\\
\end{tabularx}
\caption{Some groups for which the random search method with Corollary~\ref{int} proves useful.}
\label{table3}
\end{table}
As a consequence of the data in Table~\ref{table3}, we deduce from Corollary~\ref{int} that the groups present have greedy base size 3.
\begin{prop}\label{tabprop2}
Suppose $(G,H)$ are listed in Table~\ref{table3}. Then $\mathcal{G}(G)=3$.
\end{prop}
It now only remains to consider the action of $\mathbb{B}$ with point stabiliser $2.{}^2E_6(2).2$. Of course, as stated our probabilistic approach does not apply here since this group has base size 4, however, with a bit of work we can extend the method to this case.

\begin{prop}\label{last}
Let $G=\mathbb{B}$ with point stabiliser $H=2.{}^2E_6(2).2$. Then $\G(G)=4$.
\end{prop}

Before the proof we remind the reader of a useful definition. Let $G\leq\mathrm{Sym}(\Omega)$ and let $O_1$ and $O_2$ be orbits of $G$. The orbits $O_1$ and $O_2$ are \emph{equivalent} if for each $\alpha_1\in O_1$ there is some $\alpha_2\in O_2$ such that $G_{\alpha_1}=G_{\alpha_2}$. Define suborbits to be equivalent if the analogous condition holds.
\begin{proof}[Proof of Proposition~\ref{last}]
We estimate the probability that a randomly chosen 4-tuple contains three entries which may be ordered so that each is contained in a (fixed) largest orbit of the pointwise stabiliser of the points preceding it. We compare this with the probability that the 4-tuple forms a base, and deduce that with non-zero probability both occur simultaneously, from which the result will follow.

First note that $G$ has degree 13571955000. The subdegrees of $G$ and the corresponding 2-point stabilisers are known~\cite{stro} --- the largest two are 11174042880 with point stabiliser $K_1:=2^{1+20}. \mathrm{U}_{4}(3). 2^2$, and 2370830336 with point stabiliser $K_2:=\mathrm{Fi}_{22}.2$. We consider now the actions of $H$ on the cosets of these two point stabilisers.

We aim to get a lower bound on the largest subdegree of $H$ acting on the coset space $H/K_1$. We begin by noting that the central involution of $H$ is not in $K_1$, and so it is enough to consider the action of $E:={}^2E_6(2).2$ on $E/K_1$ --- each suborbit here corresponds to a pair of equivalent suborbits for $H$ acting on $H/K_1$. The group $E$ can be constructed within \magma~\cite{magma} as a 78-dimensional matrix group over $\mathrm{GF}(2)$. Now, $K_1$ is contained within the maximal subgroup $2^{1+20}. \mathrm{U}_{6}(2). 2<E$, which we construct as the centraliser in $E$ of a 2A-involution~\cite{atlas}. Working now in the smaller group $2^{1+20}. \mathrm{U}_{6}(2). 2$ we construct $K_1$ using a random search; we confirm the constructed group to be $K_1$ by computing the composition tree in \magma. Finally, a 2-point stabiliser in $E$ is of the form $K_1\cap K_1^g$ for some $g\in E$ --- a random search for such groups finds one of order 15552, and so $H$ has two equivalent longest suborbits on $H/K_1$, each of length at least $|K_1|/15552=1761607680$.

As was the case for $K_1$, the central involution of $H$ lies outside of $K_2$, and so the suborbits of $H$ acting on $H/K_2$ can be partitioned into pairs of equivalent suborbits. In particular, no subdegree is larger than $2370830336/2=1185415168<1761607680$. Moreover, the remaining three suborbits of $G$ each have size less than $1761607680$, and thus any largest orbit of $K_1$ acting on $G/H$ has size at least $1761607680$ and is paired with an equivalent orbit.

We now have all of the ingredients we need for the probabilistic argument. Let $d\geq 176167680$ be the size of a largest orbit of $K_1$ on $G/H$. For each $\alpha\in\Omega$ and each $\beta$ in the unique largest $G_{\alpha}$-orbit let $O_1(\alpha,\beta),O_2(\alpha,\beta),\dots,O_k(\alpha,\beta)$ be the unions of pairs of equivalent $G_{\alpha,\beta}$-orbits of size $d$. Fix some $1\leq i\leq k$. Given a 4-tuple $(\alpha_1,\alpha_2,\alpha_3,\alpha_4)$, the probability that $\alpha_2$, $\alpha_3$, and $\alpha_4$ all lie in the largest $G_{\alpha_1}$-orbit is $$\left(\frac{11174042880}{13571955000}\right)^3>558/1000,$$ and the probability that exactly two of $\alpha_2$, $\alpha_3$, or $\alpha_4$ lie in the largest $G_{\alpha_1}$-orbit is $$3\cdot\frac{11174042880^2\cdot(13571955000-11174042880)}{13571955000^3}> 359/1000.$$ Now, given $\alpha\in \Omega$ and three points $(\beta,\gamma,\delta)$ of the largest $G_{\alpha}$-orbit, the probability that either $\gamma$ or $\delta$ is in $O_i(\alpha,\beta)$ is at least $$2\cdot\frac{2\cdot1761607680}{11174042880}-\left(\frac{2\cdot1761607680}{11174042880}\right)^2>531/1000.$$ On the other hand, given $\alpha$ and a pair $(\beta,\gamma)$ of points of the largest $G_{\alpha}$-orbit, the probability that $\gamma$ is in $O_i(\alpha,\beta)$ is at least $$(2\cdot1761607680)/11174042880>315/1000.$$ We deduce that the probability that a uniformly chosen 4-tuple may be reordered as $(\alpha_1,\alpha_2,\alpha_3,\alpha_4)$ so that $\alpha_{2}$ is in the largest $G_{\alpha_{1}}$-orbit and $\alpha_3$ is in $O_i(\alpha_1,\alpha_2)$ is at least $$p_i:=(558\cdot531+359\cdot315)/10^6> 4/10.$$ Finally, the character tables of $G$ and $H$ are available in the \gap~Character Table Library~\cite{ctlib}, as is the class fusion, and so we can compute $\hat{Q}(G,4)<3/10$ directly. Therefore, the probability that a uniformly chosen $4$-tuple is a base, and may be reordered as $(\alpha_1,\alpha_2,\alpha_3,\alpha_4)$ so that $\alpha_{2}$ is in the largest $G_{\alpha_{1}}$-orbit and $\alpha_3$ is in $O_i(\alpha_1,\alpha_2)$ is at least $p_i+(1-\hat{Q}(G,4))-1>0$. That is, if $(\alpha,\beta,\gamma)$ are such that $\beta$ is in the largest $G_{\alpha}$-orbit and $\gamma\in O_i(\alpha,\beta)$, then $G_{\alpha,\beta,\gamma}$ has a regular orbit. Since this holds for each $i$ it follows that $\G(G)=4$, as desired.
\end{proof}
Combining the results of this section completes the proof of Theorem~\ref{main}.

\end{document}